\documentclass[11pt]{article} %

\usepackage{array, enumerate, color, url}
\usepackage{amsmath, epsfig}
\usepackage{latexsym}
\usepackage{mathrsfs}
\usepackage{amssymb}
\usepackage{graphicx}


 
\newtheorem{thm}{Theorem}[section]

\newcommand{\n}{\noindent}

\newtheorem{conj}[thm]{Conjecture}
\newtheorem{lemma}[thm]{Lemma}

\newenvironment{proof}{{\bf Proof}.}{\rule{3mm}{3mm}}

\title{Every planar graph without 4-cycles adjacent to two triangles is DP-4-colorable}

\author{Runrun Liu  and  Xiangwen Li\thanks{Supported in part by the NSFC (11728102) and the NSFC (11571134)}
\\
\small Department of Mathematics $\&$ Statistics\\
\small Central China Normal University, Wuhan 430079, China
}
\date{}

\begin{document}

\maketitle

\begin{abstract}
  Wang and Lih in 2002 conjectured that every planar graph without adjacent triangles is 4-choosable.  In this paper, we prove that every planar graph without any $4$-cycle adjacent to two triangles is DP-4-colorable, which improves the results of Lam, Xu and Liu [Journal of Combin. Theory, Ser. B, 76 (1999) 117--126], of Cheng, Chen and Wang [Discrete Math., 339(2016) 3052--3057] and of  Kim and Yu [arXiv:1709.09809v1].
\end{abstract}

\section{Introduction}

Coloring is one of the main topics in graph theory. A {\em proper $k$-coloring} of $G$ is a mapping $f: V(G)\to [k]$ such that $f(u)\ne f(v)$ whenever $uv\in E(G)$, where $[k]=\{1,2,\ldots, k\}$. The smallest $k$ such that $G$ has a $k$-coloring is called the {\em chromatic number} of $G$ and is denoted by $\chi(G)$. List coloring was introduced by Vizing \cite{V76},  and independently Erd\H{o}s, Rubin, and Taylor \cite{ERT79}. A {\em list assignment} of a graph $G$ is a function $L$ that assigns to each vertex $v\in V$ a list $L(v)$ of colors. An {\em $L$-coloring} of $G$ is a function $\lambda:V\to\cup_{v\in V}L(v)$ such that $\lambda(v)\in L(v)$ for every $v\in V$ and $\lambda(u)\ne\lambda(v)$ whenever $uv\in E$. A graph $G$ is {\em $k$-choosable} if $G$ has an $L$-coloring for every assignment $L$ with $|L(v)|\ge k$ for each $v\in V(G)$. The {\em choice number}, denoted by $\chi_l(G)$, is the minimum $k$ such that $G$ is $k$-choosable.

The techniques to approach the list problems are less than those used in ordinary coloring. For ordinary coloring, identifications of vertices are involved in the reduction configurations. In list coloring, since  different vertices have different lists, it is no possible for one to use identification of vertices. With this motivation, Dvo\v{r}\'ak and Postle \cite{DP17} introduced correspondence coloring (or DP-coloring) as a generalization of list-coloring. For this literature, the readers can see \cite{B16,BK17,BKZ17,KY17,LLNSY18,LLYY18}. The definitions are as follows.

 A {\em $k$-correspondence assignment} for $G$ consists of a list assignment $L$ on vertices in $V(G)$ such that $L(u)=[k]$ and a function $C$ that assigns every edge $e=uv\in E(G)$ a matching $C_e$ between $\{u\}\times[k]$ and $\{v\}\times[k]$.

A {\em $C$-coloring} of $G$ is a function $\phi$ that assigns each vertex $v\in V(G)$ a color $\phi(v)\in L(v)$, such that for every $e=uv\in E(G)$, the vertices $(u,\phi(u))$ and $(v,\phi(v))$ are not adjacent in $C_e$. We say that $G$ is {\em $C$-colorable} if such a $C$-coloring exists.

The {\em correspondence chromatic number} $\chi_{DP}(G)$ of $G$ is the smallest integer $k$ such that $G$ is $C$-colorable for every $k$-correspondence assignment $(L,C)$.

Two triangles are {\em intersecting} if they have at least one common vertex. Two triangles are {\em adjacent} if they have at least one common edge. Lam, Xu and Liu~\cite{LXL99} proved that every planar graph without 4-cycles is 3-choosable.
Wang and Lih~\cite{WL02} proved that every planar graph without intersecting triangles is 4-choosable and posed the following conjecture.

 \begin{conj}\label{conj1}(Wang and Lih, \cite{WL02})
 Every planar graph without adjacent triangles is 4-choosable.
  \end{conj}

 Conjecture~\ref{conj1} is still open.
  There have been a few results about this conjecture since 2002. Recently,  Cheng, Chen and Wang \cite{CCW16} proved that every planar graph without 4-cycle adjacent to 3-cycle is 4-choosable, which was improved to DP-4-colorable by Kim and Yu~\cite{KY17}. As we mention above,
DP-coloring is  as a generation of list coloring.  One naturally try to approach Conjecture~\ref{conj1} by utilizing DP-coloring. Motivated by this observation,  we present the following result in this paper.

\begin{thm}
\label{main}
Every planar graph $G$ without any $4$-cycle  adjacent to two triangles is DP-4-colorable.
\end{thm}

 Note that if every planar graph $G$  without any 4-cycle adjacent to two triangles, then $G$ has no adjacent triangles.  Thus, Theorem~\ref{main} generalized the results of Lam, Xu and Liu~\cite{LXL99}, of Cheng, Chen and Wang \cite{CCW16} and of Kim and Yu \cite{KY17}.  In order to show Theorem~\ref{main}, We  prove  a little stronger  theorem, as follows.

\begin{thm}
\label{main1}
Let $G$ be a plane graph without any $4$-cycle adjacent two triangles. Let $S$ be a set of vertices of $G$ such that either $|S|=1$, or $S$ consists of all vertices on a face of $G$. If $|S|\le6$, then for every $C$-coloring $\phi_0$ of $G[S]$, there exists a $C$-coloring $\phi$ of $G$ whose restriction to $S$ is $\phi_0$.
\end{thm}

\n{\bf Proof of Theorem~\ref{main}}
By Theorem~\ref{main1}, it suffices to show that $G$ is $C$-colorable for arbitrary $3$-correspondence assignment $C$.  Assume that $G$ has a assignment $C$. Take $S$ to be an arbitrary vertex in $G$.  By Theorem~\ref{main1}, $G$ is $C$-colorable.

\medskip

In the end of this section, we introduce some notations used in the paper. Graphs mentioned in this paper are all simple. Let $K$ be a cycle of a plane graph $G$. We use $int(K)$ and $ext(K)$ to denote the sets of vertices located inside and outside $K$, respectively. The cycle $K$ is called a {\em separating cycle} if $int(K)\ne\emptyset\ne ext(K)$. Let $V$ and $F$ be the set of vertices and faces of $G$, respectively. For a face $f\in F$, if the vertices on $f$ in a cyclic order are $v_1, v_2, \ldots, v_k$, then we write $f=[v_1v_2\ldots v_k]$. Let $b(f)$ be the vertex set of $f$. A $k$-vertex ($k^+$-vertex, $k^-$-vertex) is a vertex of degree $k$ (at least $k$, at most $k$).  A $k$-face ($k^+$-face, $k^-$-face) is a face contains $k$ (at least $k$, at most $k$) vertices. The same notation will be applied to cycles. Let $N(v)$ be the set of all the neighbors of $v$ and let $N[v]=N(v)\cup\{v\}$.
\section{Proof of Theorem~\ref{main1}}

This section is devoted to proving Theorem~\ref{main1}.  Let $(L,C)$ be a $k$-correspondence assignment on $G$.  An edge $uv\in E(G)$ is {\em straight}  if every $(u,c_1)(v,c_2)\in E(C_{uv})$ satisfies $c_1=c_2$.   The following lemma is from (\cite{DP17}, Lemma 7) immediately.

\begin{lemma}\label{straight}
Let $G$ be a graph with a $k$-correspondence assignment $C$. Let $H$ be a subgraph of $G$ which is a tree. Then we may rename $L(u)$ for $u\in H$ to obtain a $k$-correspondence assignment $C'$ for $G$ such that all edges of $H$ are straight in $C'$.
\end{lemma}

From now on, we always let $C$ be a $4$-correspondence assignment on $G$. Assume that Theorem~\ref{main1} fails. Let $G$ be a minimal counterexample, that is, there exists no $C$-coloring $\phi$ of $G$ whose restriction to $S$ is equal to $\phi_0$ such that
\begin{equation}
 |V(G)| \mbox{ is minimized.}
 \end{equation}
 Subject to (1), the number of edges of $G$ that do not join the vertices of $S$
 \begin{equation}
   |E(G)|-|E(G[S])| \mbox{  is minimized}.
  \end{equation}

When $S$ consists of the vertices of a face, we will always assume that $D$ is the outer face of the embedding of plane graph $G$.  A vertex $v$ or a face $f$ is {\em internal} if $v\notin D$ or $f\ne D$.

For convenience, let $F_k=\{f: \text{ $f$ is a $k$-face and } b(f)\cap D=\emptyset\}$ and $F_k'=\{f: \text{ $f$ is a $k$-face and }$
$b(f)\cap D\ne\emptyset\}$.  A $3$-face in $F_3$ is {\em{special}} if  it contains a $4$-vertex $v$ incident with at most one triangle and $N(v)\cap D=\emptyset$. Let $f$ be a $(4,4,4,4,4^+)$-face in $F_5$ adjacent to five triangles. Let $v$ be the vertex on the triangle but not on $f$. We call $f$ a {\em sink} of $v$ and $v$ a {\em source} of $f$.
The properties in  Lemma~\ref{counter} (a)-(e) is similar to \cite{DP17}.
For completeness, we include the proofs here.

\begin{lemma}\label{counter}
 Each of the following holds:
\begin{enumerate}[(a)]
\item $V(G)\ne S$;
\item $G$ is $2$-connected;
\item each vertex not in $S$ has degree at least $4$;
\item $G$ does not contain separating $k$-cycle for $3\le k\le6$;
\item $S=V(D)$ and $D$ is an induced cycle.
\item If $u$ and $v$ on $D$ are not adjacent, then they have no common neighbor not on $D$.
\item If $f$ is a sink in $G$, then at most one of its source is on $D$.
\end{enumerate}
\end{lemma}
\begin{proof}
(a) Suppose otherwise that $V(G)=S$. In this case, $\phi_0$ is a $C$-coloring of $G$, a contradiction.

(b) By the condition (1), $G$ is connected. Suppose otherwise that $v$ is a cut-vertex of $G$. Thus, we may assume that $G=G_1\cup G_2$ such that $V(G_1)\cap V(G_2)=\{v\}$.  If $v\in S$, then by the condition (1)  $G_1$ and $G_2$ have $C$-coloring extending $\phi_0$ such that these $C$-colorings have the same color at $v$.  Thus, $G$ has a $C$-coloring, a contradiction. Thus, assume that $v\notin S$. We assume, without loss of generality, that $S\subseteq V(G_1)$.  By the condition (1), $\phi_0$ can be extended to $\phi_1$ of $G_1$. Then, $\phi_1(v)$ can be extended to $\phi_2$ of $G_2$. Now $\phi_1$ and $\phi_2$ together give an extension of $\phi_0$ to $G$, a contradiction.

(c) Let $v$ be a $3^-$-vertex in $G-S$. By the condition (1), $\phi_0$ can be extended to a $C$-coloring $\phi$ of $G-v$. Then we can extend $\phi$ to $G$ by selecting a color $\phi(v)$ for $v$ such that for each neighbor $u$ of $v$, $(u,\phi(u))(v,\phi(v))\notin E(C_{uv})$, a contradiction.

(d)  Let $K$ be a separating $k$-cycle with $3\le k\le 6$. By the condition (1), $\phi_0$ can be extend to a $C$-coloring $\phi_1$ of $ext(K)\cup K$, and the restriction of $\phi_1$ to $K$ extends to a $C$-coloring $\phi_2$ of $int(K)$. Thus, $\phi_1$ and $\phi_2$ together give a $C$-coloring of $G$ that extends $\phi_0$, a contradiction.

(e) Suppose otherwise that $S=\{v\}$ for some vertex $v\in V(G)$. If $v$ is incident with a $6^-$-cycle $f_1$, we may assume that  $v$ is incident with a $6^-$-face by (d). We now redraw $G$ such that $f_1$ is the outer cycle of $G$ and choose a $C$-coloring $\phi$ on the boundary of $f_1$. Let $S_1=V(f_1)$. In this case, $|E(G_1)|-|E(G[S_1])|<|E(G)|-|E(G[S])|$. By the condition (2), $G_1$ has a $C$-coloring that extends the colors of $S_1$, thus $G$ has a $C$-coloring extends $\phi_0$, a contradiction. Thus, we may assume that all cycles incident with $v$ are $7^+$-cycles.   Let $f_2$ be a $7^+$-face  incident with $v$. Let $v_1$ and $v_2$ be the neighbors of $v$ on $f_2$. Let $G_2=G\cup \{v_1v_2\}$. We redraw $G$ such that $[vv_1v_2]$ is the outer cycle of $G_2$. Let $S_2=\{v,v_1,v_2\}$ and $C_2$ be obtained from $C$ by letting the matching between $v_1$ and $v_2$ be edgeless.  It is easy to verify that $|E(G_2)|-|E(G[S_2])|<|E(G)|-|E(G[S])|$. By the condition (2), $G_2$ has a $C_2$-coloring that extends the colors of $S_2$. This implies that $G$ has a $C$-coloring extends $\phi_0$, a contradiction again. So $S=V(D)$.

We may assume that $D$ contains a chord $uv$.  By (a) $V(G)\not=S$. Thus $D$ together with the chord $uv$ forms two cycles with common edge $uv$, each of which has of length less than 6 by our assumption that $|S|\leq 6$. By (d), such two cycles are the boundaries of two faces. This means that $S=V(G)$, a contradiction to (a).

(f) Let $u$ and $v$ be two non-adjacent vertices on $D$. Suppose otherwise that $w\notin D$ is the common neighbor of them. By (c), $d(w)\ge4$. Since $|D|\le6$, the path $P=uvw$ and $D$ form two $6^-$-cycles.  By (d), $N(w)\subset D$. But this would create a $4$-cycle adjacent to two triangles for any $D$, a contradiction.

(g) Let $u_1,u_2,\ldots,u_5$ be the five sources around $f$ in the clockwise order. Suppose otherwise that there are two sources on $D$. We first claim that at most one of $u_i$ and $u_{i+1}$ is on $D$, where the subscripts are taken modulo 5. By symmetry we may assume that $u_1$ and $u_2$ are on $D$. since a $4$-cycle in $G$ is adjacent to at most one triangle,  $u_1u_2\notin E(D)$,  $u_1$ and $u_2$ are not adjacent and $u_1$ and $u_2$ have a common neighbor not on $D$, a contradiction to (f). Thus, by symmetry, we assume that $u_1$ and $u_3$ are on $D$. By the argument above, we may assume that none of $u_2$ and $u_5$ is on $D$. Let $P$ be the path of length $3$ from $u_1$ to $u_3$ by passing two vertices on $f$. Since $|D|\le6$, let $P_1$ be the shortest path on $D$ between $u_1$ and $u_3$. Then $P_1$ has length at most $3$. Thus, $P$ and $P_1$ form a separating $6^-$-cycle, a contradiction to (d).
\end{proof}

\medskip

The following lemma plays a key role in the proof of  Theorem~\ref{main1}.

\begin{lemma}\label{reduce}
Let $v$ be a $4$-vertex and  $N(v)=\{v_i:1\le i\le4\}$ in a cyclic order. If $N[v]\cap D=\emptyset$, then each of the following holds:

(i) at most one of $v_i$ and $v_{i+2}$ of $v$ is a $4$-vertex.

(ii) If $v$ is incident with at most one triangle $f$, then $f$ must be a $(4,5^+,5^+)$-face.
\end{lemma}
\begin{proof}
Let  $f_i$ be the face  with vertices $v_i, v,v_{i+1}$ on its boundary,  where the subscripts are taken modulo $4$ for $1\le i\le4$.

(i)  We suppose otherwise that $d(v_1)=d(v_3)=4$ by symmetry. By Lemma~\ref{straight}, we can rename each of $L(v_2),L(v),L(v_4)$ to make $C_{vv_2}$ and $C_{vv_4}$ straight. Let $G'$ be the graph by identifying $v_2$ and $v_4$ of $G-\{v_1,v,v_3\}$ and let $C'$ be the restriction of $C$ to $E(G')$. Since $\{v_2,v_4\}\cap D=\emptyset$, the identification does not create an edge between vertices of $S$, and thus $\phi_0$ is also a $C'$-coloring of the subgraph of $G'$ induced by $S$. If there exists a path $P$ of length at most 4 between $v_2$ and $v_4$, then $G$ would have a separating $6^-$-cycle, contrary Lemma~\ref{counter}(d). This implies that $G'$ contains no $4$-cycles adjacent to two triangles. Also $G'$ contains no loops or parallel edges. Thus, $C'$ is also a $4$-correspondence assignment on $G'$. Since $|V(G')|<|V(G)|$,  $\phi_0$ can be extended to  a $C'$-coloring $\phi$ of $G'$  by (1). For $x\in\{v_1,v,v_3\}$, let $L^*(x)=L(x)\setminus\cup_{ux\in E(G)}\{c'\in L(x):(u,c)(x,c')\in C_{ux}$ and $(u,c)\in \phi\}$. Then $|L^*(v_1)|=|L^*(v_1)|\ge1$, and $|L^*(v)|\ge3$. So we can extend $\phi$ to a $C$-coloring of $G$ by coloring $v_2$ and $v_4$ with the color of the identifying vertex and then color $v_1,v_3,v$ in order, a contradiction.

(ii) By Lemma~\ref{counter} (c) each neighbor of $v$ has degree at least $4$.  Suppose otherwise that $f_1=[vv_1v_2]$ is not a $(4,5^+,5^+)$-face. By symmetry let $d(v_1)=4$.  By Lemma~\ref{straight}, we can rename each of $L(v_2),L(v),L(v_4)$ to make $C_{vv_2}$ and $C_{vv_4}$ straight. Let $G'$ be the graph by identifying $v_2$ and $v_4$ of $G-\{v_1,v\}$ and let $C'$ be the restriction of $C$ to $E(G')$. Since $\{v_2,v_4\}\cap D=\emptyset$, the identification does create an edge between vertices of $S$, and thus $\phi_0$ is also a $C'$-coloring of the subgraph of $G'$ induced by $S$. If there exists a path $P$ of length at most 4 between $v_2$ and $v_4$, then $G$ would have a cycle $K$ of length at most 6 which is obtained from the path $v_2vv_4$  and $P$. By our assumption, both $f_2$ and $f_3$ are $4^+$-faces. If $v_3\in P$, then $|P|\ge4$. Furthermore, if $|P|=4$, then both $f_2$ and $f_3$ are $4$-faces. Since each of $f_2$ and $f_3$ is incident to at most one triangle and since $f_1$ is a triangle, the new $4$-cycle created by $P$ in $G'$ is  adjacent to at most one triangle. If $v_3\notin P$, then $K$ is a separating $6^-$-cycle, contrary to Lemma~\ref{counter}(d). This implies that $G'$ contains no $4$-cycles adjacent to two triangles. Since $G$ contains no separating $3$-or $4$-cycle by Lemma~\ref{counter}(d), $G'$ contains no loops or parallel edges. Thus, $C'$ is also a $4$-correspondence assignment on $G'$. Since $|V(G')|<|V(G)|$,  $\phi_0$ can be extended to  a $C'$-coloring $\phi$ of $G'$  by (1). For $x\in\{v_1,v\}$, let $L^*(x)=L(x)\setminus\cup_{ux\in E(G)}\{c'\in L(x):(u,c)(x,c')\in C_{ux}$ and $(u,c)\in \phi\}$. Then $|L^*(v)|\ge2$ and $|L^*(v_1)|\ge1$. So we can extend $\phi$ to a $C$-coloring of $G$ by coloring $v_2$ and $v_4$ with the color of the identifying vertex and then color $v_1$ and $v$ in order, a contradiction.
\end{proof}

\medskip

We are now ready to present a discharging procedure that will complete the proof of the Theorem~\ref{main1}.  Let each vertex $v\in V(G)$ have an initial charge of $\mu(v)=2d(v)-6$,  each face $f\not=D$ have an initial charge of $\mu(f)=d(f)-6$, and $\mu(D)=d(D)+6$. By Euler's Formula, $\sum_{x\in V\cup F}\mu(x)=0$. Let $\mu^*(x)$ be the charge of $x\in V\cup F$ after the discharge procedure. To lead to a contradiction, we shall prove that $\mu^*(x)\ge 0$ for all $x\in V\cup F$ and $\mu^*(D)$ is positive.

\noindent Let $t$ be the number of incident triangles of $v$. The discharging rules are as follows.

\begin{enumerate}[(R1)]
\item \label{4vertex} Let $v\notin D$ be a $4$-vertex.
\begin{enumerate}[(a)]
\item \label{41} Let $t\le1$. If $N(v)\cap D=\emptyset$, then $v$ gives $\frac{1}{2}$ to each incident face. If $N(v)\cap D\ne\emptyset$, then $v$ gives $1$ to each incident $3$-face, $\frac{1}{2}$ to each face in $F_4$ or $F_5$ and give its rest charge evenly to other incident faces.

\item \label{42} If $t=2$, then $v$ gives $1$ to each incident $3$-face.
\end{enumerate}

\item\label{5vertex} Let $v\notin D$ be a $5$-vertex.
\begin{enumerate}[(a)]
\item \label{51} If $N(v)\cap D=\emptyset$, then $v$ gives $\frac{5}{4}$ to each special $3$-face, $1$ to each other $3$-face, $\frac{1}{2}$ to each other incident face and $\frac{1}{4}$ to each sink.

\item \label{52} If $N(v)\cap D\ne\emptyset$, then $v$ gives $\frac{5}{4}$ to each incident $3$-face, $\frac{1}{2}$ to each incident face in  $F_4$ or $F_5$, $\frac{1}{4}$ to each sink and gives the rest charge evenly to other incident faces.
\end{enumerate}

\item \label{6+} Each $6^+$-vertex not on $D$ gives $\frac{5}{4}$ to each incident $3$-face, $\frac{1}{2}$ to each other incident face and $\frac{1}{4}$ to each incident sink.

\item \label{D}Each vertex on $D$ gives its initial charge to $D$ and $D$ gives $2$ to each face in $F_3'$, $\frac{7}{4}$ to each other face in $F_k'$ for $k\ge4$.

\end{enumerate}

\begin{lemma}\label{CHECK}
Every vertex $v$ and internal face $f$ in $G$ have nonnegative final charges.
\end{lemma}
\begin{proof}
We first check the final charges of vertices. Let $v$ be a vertex in $G$. If $v\in D$, then  $\mu^*(v)\ge0$ by (R\ref{D}). Thus, we may assume that  $v\notin D$. By Lemma~\ref{counter}(c), $d(v)\ge4$. Let $d(v)=4$. Recall that $t$ is the number of incident $3$-faces of $v$. Since $G$ contains no adjacent triangles, $t\le2$. Let $t\le1$.  If $N(v)\cap D=\emptyset$, then $v$  gives  $\frac{1}{2}$ to each incident face by (R\ref{41}).  If $N(v)\cap D\ne\emptyset$, then $v$ sends $1$ to each incident $3$-face, $\frac{1}{2}$ to each face in $F_4$ or $F_5$ and give its rest charge evenly to other incident faces. Furthermore, if $t=1$, then $v$ is incident with at most two faces from $F_4$ or $F_5$. Thus, $\mu^*(v)\ge2\times4-6-\max\{\frac{1}{2}\times4, 1+\frac{1}{2}\times2\}=0$. If $t=2$, then by  (R\ref{42})$v$ gives $1$ to each incident $3$-face. So $\mu^*(v)\ge2\times4-6-1\times2=0$.
Next, assume that $d(v)=5$. We first assume that $N(v)\cap D=\emptyset$. By (R\ref{51}) $v$ gives $\frac{5}{4}$ to each special $3$-face, $1$ to each other $3$-face, $\frac{1}{2}$ to each other incident face and $\frac{1}{4}$ to each sink. Let $f$ be a special $3$-face incident with $v$. Note that the $4$-vertex on $f$ is incident with at most one triangle. Thus, the face  adjacent to $f$ but $v$ not on its boundary is not a sink. So $\mu^*(f)\ge2\times5-6-\max\{\frac{5}{4}, 1+\frac{1}{4}\}\times2-\frac{1}{2}\times3=0$. If  $N(v)\cap D\ne\emptyset$, then by (R\ref{52}) $v$ gives $\frac{5}{4}$ to each incident $3$-face, $\frac{1}{2}$ to each incident face in  $F_4$ or $F_5$, $\frac{1}{4}$ to each sink and gives the rest charge evenly to other incident faces. If $t\le1$, then $\mu^*(v)\ge2\times5-6-\frac{1}{4}-\frac{5}{4}-\frac{1}{2}\times4>0$. If $t=2$, then $v$ is incident with at most two faces in $F_4$ or $F_5$. So $\mu^*(v)\ge2\times5-6-2(\frac{1}{4}+\frac{5}{4})-\frac{1}{2}\times2=0$.  Finally, assume that $d(v)\ge6$. By (R\ref{6+}) $v$ gives $\frac{5}{4}$ to each incident $3$-face, $\frac{1}{2}$ to each other incident face and $\frac{1}{4}$ to each incident sink. So $\mu^*(v)\ge2d(v)-6-(\frac{5}{4}+\frac{1}{4})\lfloor\frac{d(v)}{2}\rfloor-\frac{1}{2}\lceil\frac{d(v)}{2}\rceil\ge d(v)-6\ge0$.

Now we check the final charge of internal faces. Since $6^+$-faces are not involved in the discharging procedure, they have non-negative final charge. Let $f$ be a $5^-$-face in $G$.  We first assume that $V(f)\cap D=\emptyset$. Let  $d(f)=3$. If $f$ is special, then by Lemma~\ref{reduce}(ii) $f$ is a $(4,5^+,5^+)$-face. By(R\ref{4vertex}), (R\ref{5vertex}) and (R\ref{6+}) $f$ gets $\frac{1}{2}$ from the incident $4$-vertex and $\frac{5}{4}$ from each incident $5^+$-vertex. If not, then by  (R\ref{4vertex}), (R\ref{5vertex}) and (R\ref{6+}) $f$ gets at least $1$ from each incident vertex. In either case, $\mu^*(f)\ge3-6+\min\{\frac{1}{2}+\frac{5}{4}\times2, 1\times3\}=0$. If $d(f)=4$, then $f$ is adjacent to at most one triangle. So by (R\ref{4vertex}), (R\ref{5vertex}) and (R\ref{6+}) $f$ gets $\frac{1}{2}$ from each incident vertex. So $\mu^*(f)\ge4-6-\frac{1}{2}\times4=0$. For $d(f)=5$, let $f=[v_1v_2\ldots v_5]$. By (R\ref{4vertex}), (R\ref{5vertex}) and (R\ref{6+}) $f$ gets $\frac{1}{2}$ from each incident $5^+$-vertex or $4$-vertex that is incident with at most one triangle. If $f$ is incident with at least two $5^+$-vertices or $4$-vertices that is incident with at most one triangle, then $\mu^*(f)\ge5-6-\frac{1}{2}\times2=0$. So we may assume that $f$ is a sink. By Lemma~\ref{counter}(g) at most one of the five sources of $f$ is on $D$ and by Lemma~\ref{reduce}(i) each source not on $D$ is a $5^+$-vertex. So by (R\ref{5vertex}) and (R\ref{6+}) each source not on $D$ gives $\frac{1}{4}$ to $f$. So $\mu^*(f)\ge5-6-\frac{1}{4}\times4=0$.

Now we assume that $V(f)\cap D\ne \emptyset$. If $d(f)=3$, then at least one vertex on $f$ is not on $D$ by Lemma~\ref{counter}(e), which gives $1$ to $f$ by (R\ref{4vertex}), (R\ref{5vertex}) and (R\ref{6+}). By (R\ref{D}) $D$ gives $2$ to $f$. So $\mu^*(f)\ge3-6+2+1\ge0$. If $d(f)=4$, then by Lemma~\ref{counter}(e) and (f) $f$ either share one vertex with $D$ or one edge with $D$.  Let $f=[v_1v_2v_3v_4]$. In the former case, by symmetry we may assume that $V(f)\cap D=\{v_1\}$. Since $f$ is adjacent to at most one triangle, at least one of $v_1v_2$ and $v_1v_4$ is not on a $3$-face, say $v_1v_2$. Then $f$ gets at least $\frac{2-1-\frac{1}{2}}{2}=\frac{1}{4}$ from $v_2$ if $d(v_2)=4$ by (R\ref{4vertex}), $\frac{4-\frac{5}{4}\times2-\frac{1}{4}\times2-\frac{1}{2}}{2}=\frac{1}{4}$ from $v_2$ if $d(v_2)=5$ by (R\ref{5vertex}), and $\frac{1}{2}$ from $v_2$ if $d(v)\ge6$ by (R\ref{6+}). In addition, by (R\ref{D}) $f$ gets $\frac{7}{4}$ from $D$. So $\mu^*(f)\ge4-6+\frac{7}{4}+\frac{1}{4}=0$. If $d(f)=5$, then $f$ gets $\frac{7}{4}$ from $D$ by by (R\ref{D}). So $\mu^*(f)\ge5-6+\frac{7}{4}>0$.
\end{proof}

\medskip

\n{\bf Proof of Theorem~\ref{main1}}. By Lemmas~\ref{CHECK}, it is sufficient for us to check that the outer face $D$ has positive final charge. Let $E(D, V(G)-D)$ be the set of edges between $D$ and $V(G)-D$ and let $e(D, V(G)-D)$ be its size. By (R\ref{D}) we have
\begin{align}
\mu^*(D)&=d(D)+6+\sum_{v\in D} (2d(v)-6)-2|F_3'|-\frac{7}{4}|F_k'|\\
&=d(D)+6+2\sum_{v\in D} (d(v)-2)-2d(D)-2|F_3'|-\frac{7}{4}|F_k'|\\
&=6-d(D)+\frac{1}{8}e(D,V(G)-D)+\frac{15}{8}e(D,V(G)-D)-2|F_3'|-\frac{7}{4}|F_k'|
\end{align}
where $k\ge4$.

So we may think that each edge $e\in E(C,V(G)-C)$ carries a charge of $\frac{15}{8}$. Then let $f$ be a face including $e$. Let $e$ give $1$ to $f$ if $f$ is a $3$-face, $\frac{7}{8}$ to $f$ otherwise.
Since $G$ contains no adjacent triangles and each face in $F_k'$ for $k\ge3$ contains two edge in $E(C,V(G)-C)$, this implies that $\frac{15}{8}e(D,V(G)-D)-2|F_3'|-\frac{7}{4}|F_k'|\ge0$. By Lemma~\ref{counter}(a) $e(D,V(G)-D)>0$. Then $\mu^*(D)>0$ for any $D$.

\end{document}